\documentclass[a4paper, 12pt]{article}
\usepackage{graphics}
\usepackage{ifthen}
\usepackage{amsmath,amsfonts,amstext,amscd,bezier,amssymb,a4,enumerate,epsfig,psfrag,subfigure}
\usepackage[latin1]{inputenc}
\usepackage{graphicx}

\usepackage{amsthm,amsopn}

\usepackage{color} 

\everymath={\displaystyle}

\newcommand{\NN}{{\mathbb{N}}}
\newcommand{\MM}{{\mathbb{M}}}

\newcommand{\RR}{{\mathbb{R}}}

\newcommand{\CCC}{{\cal{C}}}
\newcommand{\DDD}{{\cal{D}}}

\newcommand{\ub}{{\stackrel{{ub}}{\longrightarrow}}}

\newcommand{\hY}{{{H^0}}}

\newcommand{\bY}{{{B^0}}}
\newcommand{\Hzb}{{{Z}}}
\newcommand{\e}{{\, = \,}}

\newtheorem{lemma} {Lemma}
\newtheorem{prop} {Proposition}
\newtheorem{theo} {Theorem}

\newtheorem*{theo*} {Theorem}

\renewcommand{\qed}{\hfill \mbox{\raggedright \rule{.1in}{.1in}}}

\newcommand{\Ran}{\operatorname{Ran}}
\newcommand{\Ker}{\operatorname{Ker}}

\title{Abundance of cusps and a converse to the Ambrosetti-AProdi theorem}
\author{Marta Calanchi, Carlos Tomei and André Zaccur}

\date{}
\begin{document}

\maketitle

\centerline{\it To Louis Nirenberg, wonderful example}

\begin{abstract}
According to the Ambrosetti-Prodi theorem, the map $F(u)= - \Delta u - f(u)$ between appropriate functional spaces is a global fold. Among the hypotheses, the convexity of the function $f$ is required. We show in two different ways that, under mild conditions, convexity is indeed necessary. If $f$ is not convex, there is a point with at least four preimages under $F$. More, $F$ generically admits cusps among its critical points. We present a larger class of nonlinearities $f$ for which the critical set of $F$ has cusps. The results are true for a class of boundary conditions.
\end{abstract}

\medbreak


\smallbreak

{\noindent\bf MSC-class:} 35B32, 35J91, 65N30.

\bigskip
\bigskip

The celebrated Ambrosetti-Prodi theorem \cite{AP}, originally a statement about a differential operator between H\"older spaces, received a number of amplifications and  formulations from different authors.
Manes and Micheletti \cite{MM} weakened the original hypothesis, and Berger and Podolak \cite{BP}, together with later work of Berger and Church \cite{BC}, presented a geometric phrasing on Sobolev spaces. We present a version for each scenario.

For an open bounded  domain $\Omega \subset \RR^n$ with piecewise smooth boundary,
we consider the H\"older spaces $B^2_D= C^{2,\alpha}_0(\Omega),  \bY = C^{0,\alpha}(\Omega), \ \alpha \in (0,1)$, and Sobolev spaces $H^2_D = H^1_0(\Omega) \cap H^2(\Omega),  \, \hY  = H^0(\Omega)= L^2(\Omega)$.
The  Dirichlet Laplacian
\[ -\Delta_D:H^2_D \subset \hY  \to \hY  \, , \quad  \sigma (-\Delta_D) = \{\, 0 < \mu_{1,D} < \mu_{2,D} \le \cdots \, \} \]
has pure point spectrum $\sigma (-\Delta_D)$ and the eigenvalues $ \{ \mu_{k,D} \}$ are associated with a complete set of $L^2$ orthonormal eigenfunctions $\{ \psi_{k,D} \}$ in $B^2_D$. Let $f: \RR \to \RR$ be a strictly convex smooth function with
\[ \overline{f'(\RR)} = [a,b], \
 a,b \notin \sigma(-\Delta_D), \quad (a,b) \cap \sigma(-\Delta_D) = \{ \mu_{1,D}\}, \quad \lim_{|x| \to \infty} f''(x) =0.  \]

\begin{theo*}[Ambrosetti-Prodi, \cite{AP}, \cite{MM}] The differential operator
\[F: B^2_D \to \bY, \quad F(u) =  -\Delta_D u - f(u) \]
is a smooth map with  critical set $\CCC \subset B^2_D$ diffeomorphic to a hyperplane. The complement $\bY - F(C)$ splits in two connected components $\CCC_0$ and $\CCC_2$. The sets $\CCC_0, F(C)$ and $\CCC_2$ have respectively zero, one and two preimages.
\end{theo*}

The critical set of $F$ consists of the critical points of $F$ or, in other words, the functions $u \in B^2_D$ for which $DF(u)$ is not an isomorphism.

Berger and Podolak \cite{BP} and Berger and Church \cite{BC} introduced additional geometric ingredients. Here,  Sobolev spaces are especially convenient. For an element $u$ of a vector space, the set $\langle u \rangle$ is the line through $0$ and $u$.

\begin{theo*}[Berger-Church-Podolak, \cite{BP}, \cite{BC}] \label{theo:BP} For $f$ as above, $F:H^2_D \to\hY $ is a {\it global fold}. More precisely, split $\hY  = \ W^0 \oplus \langle \psi_{1,D} \rangle  $ in orthogonal subspaces. Then there are global homeomorphisms $\zeta:H^2_D \to W^0 \oplus \RR  $ and $\xi:\hY   \to W^0 \oplus\RR  $ for which $\tilde{F}(z,t) \, = \, \xi \circ F \circ \zeta^{-1} (z,t) = (z,t^2) \, .$
\end{theo*}

Said differently, the following diagram commutes.
\[
\begin{array}{ccl}
H^2_D& \stackrel{{\scriptstyle F}}{\longrightarrow}&\hY\\
   {\scriptstyle \zeta}\downarrow & &
\downarrow{\scriptstyle \xi}\\
 W^0 \oplus \RR  &\stackrel{{\scriptstyle (z, t^2)}}{\longrightarrow}&   W^0 \oplus \RR   \\
  \end{array}
\]

So-called Ambrosetti-Prodi type results for different boundary conditions have been considered also extensively (\cite{CMN},\cite{O},\cite{MRZ}, \cite{MP}).

Settling a question raised by Dancer \cite{D}, we prove a converse result under mild conditions: the theorems do not hold for a nonconvex function $f$, for a number of boundary conditions. More precisely, we consider a domain of self-adjointess $H^2_b(\Omega)\subset H^0(\Omega)$ of $-\Delta_b: H^2_b(\Omega)  \to H^0(\Omega)$ corresponding to {\it standard boundary conditions}, defined in Section \ref{section:spectrum}, for which $\sigma(-\Delta_b) = \{ \mu_{1,b} < \mu_{2,b} \le  \ldots \}$.
Dirichlet, Neumann and periodic boundary conditions are standard.

As for the nonlinearity, we assume that $f: \RR \to \RR$ is a smooth function and, in different sections, we use some of the following  hypotheses:

\begin{itemize}
\item[(1)] (strict interaction with $\mu_{1,b}$) There are $m, M \in \RR$ for which
\[ \overline{f'(\RR)} = [m,M] \quad \hbox{and} \quad m < \mu_{1,b} < M < \mu_{2,b} \, ; \]
\item[(2)] (genericity) The functions $f' - \mu_{1,b}$, $f''$  and $f'''$ have no common zero;
\item[(3)] (nonconvexity) $f''$ has a zero $x_\ast$;
\item[(4)] (nonresonance with $\mu_{1,b}$) There is an $\epsilon >0$ such that
\[ f'(x) > \mu_{1,b} + \epsilon \ \hbox{ for } \ x \to \infty \ \hbox{ and } \ f'(x) < \mu_{1,b} - \epsilon \ \hbox{ for }\ x \to -\infty. \]
\end{itemize}

Define $B^2_b = H^2_b \cap C^{2,\alpha}(\Omega)$ and consider $F: B^2_b \to B^0$ and let $D$ be a dense subspace of $B^2_b$ in the $L^2$ norm.

\begin{theo}\label{Rec} Consider a standard boundary condition and let  $f:\RR \to \RR$ satisfy hypotheses $(1), (2), (3)$ and $(4)$. Then for some $y \in B^0 $, $F(u)=y$ admits (at least) four solutions in $D \subset B^2_b$.
\end{theo}

The result is proved in Section \ref{section:theo3}.
The proof uses ideas of Berger and Podolak, also  extensively used in \cite{MST}. We consider {\it fibers}, inverses under $F$ of special straight lines, described in more detail in Section \ref{section:fibers}. Under  hypotheses $(1)$ and $(4)$, for appropriate coordinates, the restriction of $F$ to a fiber is a map from $\RR$ to $\RR$ which goes to $-\infty$ for $|t| \to \infty$. The existence of a point with four preimages reduces to the search of a fiber on which such map admits a local minimum. Only a weaker version of genericity (hypothesis $(2)$) is used in the proof.

\medskip
In Section \ref{section:cusp1}, we provide geometric information of local character which, in a sense, strengthens the previous theorem: generically, $F$ admits {\it cusps}. The literature concerning cusps of differential operators is extensive (\cite{CaT}, \cite{ChT2}, \cite{LM}, \cite{R1}, \cite{R2}). A self-contained definition is the following.
Let $\lambda_1(u)$ be the smallest eigenvalue and $\phi_1(u)$ the associated $L^2$-normalized eigenvector of the Jacobian $DF(u): B^2_b \to B^0$. As we shall see in Proposition \ref{prop:taubasta}, a zero $u_c$ of the function
\[ \Lambda: B^2_b \to \RR^2 \, , \quad \Lambda(u) = (\lambda_1(u), \ \delta_1(u) = \ D \lambda_1(u) \ \phi_1(u)    )  \]
with $\tau_1(u_c) = D\delta_1(u_c) \, \phi_1(u_c) \ne 0 $ is a cusp of $F: B^2_b \to \bY$.

At a cusp $u_c$, a function $F$ admits a simple local form \cite{ChT2}: for some Banach space $Y$, changes of variables near $u_c$ and $F(u_c)$ convert $F$ into
\[ \tilde F (w,x,y ) \, = \, (\, w, x, y^3 - x \, y ) \, , \quad \hbox{for} \, \, w \, \in Y \, , \ x , \, y \, \in \, \RR \,  . \]
In particular, points near $F(u_c)$ may have one, two or three preimages near $u_c$.

\begin{theo}\label{theo:c} Consider a standard boundary condition. Let $f$ satisfy  $(1), (2), (3)$.
Then  either $F: B^2_b \to \bY$ has a cusp in $D$ or there is a family of disjoint arcs  each of which is taken by $F$ to a single point. In both cases, for some $g \in B^0$, $F(u) = g$ has at least three solutions. If $F$ is proper, it has a cusp in $D$.
\end{theo}

An arc is a diffeomorphic image of an open interval.
If collapsing arcs exist, they  are abundant, being parameterized by an open set of a codimension 3 subspace in $B^2_b$.
There is another simple hypothesis which guarantees the existence of a cusp, presented in Proposition \ref{prop:noworse}.

An ancestor of these results is Theorem 4.3 in \cite{MST}. Suppose $f(x)$ goes to $\infty$ when $|x|\to \infty$ and $f''$ is strictly positive: then the differential operator
\[ G: C^1 ([0,1]) \to C^0([0,1]), \quad u \mapsto u' + f(u), \quad u(1) = u(0), \]
is a global fold. On the other hand, under generic hypotheses which we do not describe, if $f'' $ is negative at some point then $G$ has points with four preimages.

Cusps are frequently associated with operators $F(u) = - \Delta u - f(u)$ for cubic nonlinearities $f$. Ruf's global study of the geometry of the nonlinearity $f(u) = -u^3 + c u$  is a very interesting case study (\cite{R1}), in which the associated operator $F$ is proper.  Hypothesis $(1)$ excludes such functions, but the alternative hypothesis $(H_k)$ is  more tolerant:
\begin{itemize}
\item[$(H_k)$] {The operator $-\Delta_b: H^2_b \to H^0$ is self-adjoint, some eigenvalue $\mu_{k,b}$  is isolated and simple and  there are points $x_\mu, y_\mu \in \RR$ for which
\[ f'(x_\mu) = f'(y_\mu)= \mu_{k,b} \quad \hbox{and} \quad f''(y_\mu) \  f''(x_\mu) < 0 \, . \]
}
\end{itemize}

\begin{theo}\label{theo:3sol} Suppose standard boundary conditions and hypotheses $(2)$ and $(H_k)$. Then either $F: B^2_b \to \bY$ has a cusp in $D$ or there is a family of disjoint arcs  each of which is taken by $F$ to a single point.  For $k=1$, a cusp in $D$ necessarily occurs  if $F$ is proper or $f'''(x_\mu), f'''(y_\mu) \ge 0$.
\end{theo}

The proof of Theorems \ref{theo:c} and \ref{theo:3sol} splits into a few steps. We first show in Section \ref{section:defcusp} that the requirements on $u_c$ in Theorem  \ref{theo:c} consist of an appropriate description of a cusp. In Section \ref{section:nonfold1} we find a zero $\overline u_{nf}$ of (an extension of) $\Lambda$ taking only two real values, which is mollified in Section \ref{section:mollifiers} to obtain a smooth zero $u_{nf}$ of $\Lambda$ in $D$. We are left with showing in Section \ref{section:transversality} that the transversality conditions stated in Theorem \ref{theo:c} are satisfied either by $u_{nf}$ or by some zero $u_c \in D$ of $\Lambda$ nearby. For Dirichlet conditions, for example, we may take $u_c \in C_0^\infty(\Omega)$.

The technique for mollifying functions respecting nonlinear restrictions used to pass from $\overline u_{nf}$ to $u_{nf}$ might be of independent interest: a different version was used in the construction of homotopies in \cite{BST1}.

%

Under hypotheses $(1)-(4)$, the existence of a cusp $u_c$ implies Theorem \ref{Rec}. Indeed, from the local form of $F$ near $u_c$, there is a point $g$ with three preimages, as for the polynomial $p(x) = x^3 - x$ near zero. A fourth pre-image arises  because, by hypothesis $(4)$, $F$ is proper of degree zero (or because along a fiber, for large $|x|$, the function $F$ looks like $x \mapsto - x^2$).

\medskip
Theorems \ref{Rec}  and  \ref{theo:c}  replicate the structure of a pair of papers by Ruf. In \cite{R3} he finds points in the image of a semi-linear elliptic boundary value operator  with five preimages. In \cite{R4} he shows that the operator acting on functions defined on intervals or rectangles (with Neumann boundary conditions) admits a {\it butterfly} $u_b$, so that there are points near $F(u_b)$ with five preimages. Since we stop at cusps, our computations are simpler despite of the fact that we handle Laplacians on arbitrary bounded sets. The mollifying arguments we employ allow us to be rather careless about boundary conditions up to the last moment.

\smallskip
\noindent{ \bf Acknowledgements} Tomei and Zaccur gratefully acknowledge support from CAPES, CNPq and FAPERJ.

\section{Proof of Theorem  \ref{Rec}} \label{section:theo3}

We sketch the proof of Theorem  \ref{Rec}. Let $\psi_{1,b}$ be the positive normalized eigenvector associated with the free eigenvalue $\mu_{1,b}$. In Section \ref{section:fibers}, following ideas in \cite{BP}, we foliate $H^2_b$ and  $B^2_b$ into {\it fibers}, the inverses of lines $\{z+ s \, \psi_{1,b} \, , s \in \RR \}$ which turn out to be
connected curves of the form
$\{ w(z,t) + t \psi_{1,b} , \, t \in \RR, \langle w(z,t), \psi_{1,b} \rangle = 0 \}$ (brackets denote the usual $L^2$ inner product).

Fibers are easier to obtain in Sobolev spaces, but in general working on H\"older spaces is simpler, due to the additional smoothness of the function $F$ restricted to them. The restriction $F^z$ of $F$ to each fiber $\{ u(z,t), t \in \RR\} \subset B^2_b$ is of the form  $F^z(u(z,t)) = z + h(u(z,t)) \psi_{1,b} $ for a real valued {\it height } $h$ for which, as shown in Proposition \ref{prop:infinity},
\[ \lim_{ |t| \to \infty} h(u(z,t)) = - \, \infty .\]

Thus, in order for $F$ to be a global fold, the restriction $F^z$ should look (topologically) like $t \mapsto - t^2$. On the other hand, we will construct in Section \ref{subsection:psi} fibers $u(z,t)$ on which $F^z$ admits a strict local minimum. The asymptotic behavior of $h$ implies the existence of points with four preimages, proving Theorem \ref{Rec}. The interested reader may find a numerical example in Section 5.3 of \cite{CTZ1}.

It turns out that critical points of $h$ along a fiber are exactly the critical points of $F$ (Proposition \ref{prop:hcrit}). The properties of interest at the local minimum $u_m$,
\begin{enumerate}
\item $u_m$ is such that $v \mapsto - \Delta_b \, v - f'(u) v$ has an eigenvalue equal to 0,
\item The height $h$ at $u_m$ has positive second  derivative along its fiber,
\end{enumerate}
are verified by checking if $ \lambda_1(u ) = 0 , \, \delta_1(u)>0$ for
appropriate functionals $\lambda_1$ and $\delta_1$ (introduced in
Theorem \ref{theo:c}) which extend to bounded functions in $H^0$. We first find such a point $\overline u$ among {\it two
valued potentials}, a class of very simple functions defined in
Section \ref{four}. Mollification then yields the required  $u \in
B^2_b$.

\subsection{Basic spectral theory and smoothness} \label{section:spectrum}

We consider  boundary conditions associated with domains $H^2_b = H^2_b(\Omega) \subseteq H^2(\Omega)$ on which $-\Delta_b: H^2_b  \to H^0$ is self-adjoint. Set
$ \sigma(-\Delta_b) = \{ \mu_{1,b} < \mu_{2,b} \le \ldots\} $
and let $\psi_{1,b}$ be the positive normalized eigenfunction associated with $\mu_{1,b}$.
By the Kato-Rellich theorem, for $q\in L^\infty(\Omega)$,
\[ T_q : H^2_b \subset \hY \to \hY, \quad v \mapsto - \Delta_b \, v - q \, v \]
is also self-adjoint. A boundary condition is {\it standard} if the  conditions below hold.
\begin{enumerate}
\item The smallest eigenvalue $\lambda^q_1$ of $T_q$
is  simple (hence isolated).
\item There is a unique  $L^2$-normalized  eigenfunction $\phi^q_1 > 0$ associated with  $\lambda^q_1$.
\item  On  bounded sets of potentials $q$, the sup norm of $\phi^q_1$ is uniformly bounded.
\end{enumerate}
The boundary condition is implicit in the notation  $\lambda^q_1$ and $\phi^q_1$. We recall some basic facts from  spectral theory and elliptic regularity (\cite{GT}, \cite{N}).

\begin{prop} \label{prop:lambdacont} Dirichlet, Neumann and periodic boundary conditions are standard. For standard boundary conditions and $q \le M < \mu_{2,b}$, $T_q$ is not invertible if and only if $\lambda^q_1 = 0$.
\end{prop}

Set $B^2_b= C^{2,\alpha}(\Omega) \cap H^2_b$.
The differentiability of $F$ is also well known (\cite{AP2},\cite{ChT2}).

\begin{prop} \label{prop:smooth}  For a smooth function $f: \RR \to \RR$ the map $F:B^2_b \to \bY$ is smooth. If $f$ satisfies hypothesis $(1)$, then
$F:H^2_b \to\hY $ is a $\CCC^1$ map. In both cases, if $f'(u) \le M < \mu_{2,b}$, the differential $DF(u) v = -\Delta_b v - f' (u) v$  is always a Fredholm operator of index zero, with kernel of dimension at most one.
\end{prop}

In $\Hzb = L^2(\Omega) \cap L^\infty(\Omega)$ consider the  $L^2$-inner product $\langle u, v \rangle$ --- notice that $Z$ is not a Banach space. A sequence $\{ u_m, m \in \NN\} \subset \Hzb$ is {\it $ub$-convergent}, $ u_m \, \ub \ u_\infty \,  , $
if $u_m \to u_\infty$ in $L^2$ and $\{ u_m, m \in \NN\}$ is bounded in the $L^\infty$ norm. Given a metric space $\MM$, a function  $G: X \subset \Hzb \to \MM$ is $ub$-continuous if it takes $ub$-convergent sequences  to convergent sequences in $\MM$. In particular, if $X$ is bounded in the sup norm, $G$ is continuous.
\bigskip

For a smooth function $f: \RR \to \RR$ we are interested in  potentials of the form $q = f' (u)$. For a standard boundary condition, let $\lambda_1(u) = \lambda^q_1$ denote the smallest eigenvalue of the Jacobian $DF(u): H^2_b \to H^0$ and $\phi_1(u) = \phi^q_1$ is the associated positive $L^2$-normalized eigenvector, which by standard regularity results is necessarily in $B^2_b$ if $u$ is.
\medskip

\begin{prop} \label{prop:fantasma}
The functions  $\lambda_1: B^2_b \to \RR$ and $\phi_1: B^2_b \to \hY$ are smooth. The extensions
$\lambda_1: \Hzb \to \RR $ and  $\phi_1: \Hzb \to \hY $ are $ub$-continuous.
\end{prop}

\medskip
\noindent The proof is given in Section \ref{subsection:spectrumu}.
\medskip

For $F$ between Sobolev spaces, the functions $\lambda_1$ and $\phi_1$ are not  smooth enough to consider higher singularity theory, in particular cusps.  Sobolev spaces suffice in Section \ref{section:theo3}, which handles folds in a disguised form. For $F$ between H\"older spaces, the smoothness of $\lambda_1$ and $\phi_1$ is used in Section \ref{section:cusp1} when we study cusps.

\subsection{Fibers and asymptotics on fibers} \label{section:fibers}

Locally, the construction of fibers is a Lyapunov-Schmidt decomposition associated with an eigenvector.
Hypothesis $(1)$ provides a {\it global} decomposition.
There are analogous results which hold locally under hypothesis $(H_k)$, and this fact will be used in Section \ref{section:cusp2}. The arguments in this section are valid for boundary conditions for which the smallest eigenvalue of $-\Delta_b: H^2_b \to H^0$ is isolated and simple --- the positivity of the ground state is not needed.

Split $H^2_b$ and $H^0$   orthogonally  into {\it horizontal} and {\it vertical} subspaces,
\[H^2_b =  W^2 \oplus V \ , \quad\hY  =  W^0  \oplus V , \quad V = \langle \psi_{1,b} \rangle, \quad \| \psi_{1,b}\|_\hY =1 \ . \]
For a fixed $z  \in W^0$, the set $\{ z+ s \, \psi_{1,b}\, ,   z \in W^0\, , \ s \in \RR \}$ is a vertical line in the image --- its inverse  under $F$ is a $\it fiber$ $F^z$.
Clearly, the domain $H^2_b$ is a disjoint union of fibers.
Versions of the next result may be found in \cite{BP}, \cite{P}, \cite{ChT2},  \cite{TZ}.

\begin{prop}\label{prop:smoothfiber} Suppose hypothesis $(1)$.
Fibers $F^z$ are indexed by $z \in W^0$ and are parameterized by $t \in \RR$: $u(z,t) = w(z,t) + t \psi_{1,b}$, for a $\CCC^1$ map $(z,t) \mapsto w(z,t) \in W^2$. The map $F: B^2_b \to B^0$  admits similar {\it smooth} fibers: a fiber of $F:H^2_b \to H^0$ with a point in $B^2_b$ is completely in $B^2_b$.
\end{prop}

The last statement seems to be new and we sketch a proof: at regular points of $F: B^2_b \to B^0$, tangent vectors to the fibers are inverses of $\psi_{1,b} \in B^0$ under $DF(u): B^2_b \to B^0$. One can normalize this vector field and extend it smoothly to the whole space $B^2_b$ --- fibers in $B^2_b$ are the orbits of this vector field.
\medskip

We now consider $F$ along a fiber $F^z = \{u(z,t) \} \subset B^2_b$ ---  the parameter $z$ is kept fixed, and we drop any notational reference to it. The restriction  of $F$ to a fiber is essentially given by the {\it height} $h$ of its images, \[ F(u(t)) = z + h(u(t)) \, \psi_{1,b}, \quad t \in \RR \ ,\]
a smooth map for fibers in $B^2_b$.

The critical set of $F:B^2_b \to B^0$ restricts well to fibers.
More, critical points of the height are described in terms of spectrum of $DF$ (\cite{BP},\cite{ChT2}).

\begin{prop} \label{prop:hcrit}  Let $f: \RR \to \RR$ be smooth and suppose $(1)$.
The derivative $h(u(t))$,  the height of a fiber $u(t)$,   is zero exactly at critical points $u(t_0)$ of $F: B^2_b \to B^0$. The eigenfunction $\phi_1(u(t_0))$ of $DF(u(t_0))$ is a positive multiple of $u'(t_0)$, the tangent vector to the fiber at $u(t_0)$. Finally, there is a strictly positive smooth function $p: B^2_b \to \RR$ for which
\[ \frac{d}{d  t}\,  \, h(u(t)) \,  \e  Dh(u(t)) \, u'(t) \e p(u(t)) \ \lambda_1(u(t)) \, . \]
\end{prop}

\begin{proof} Differentiate the formula for $h$, $DF(u(t)) u' (t) = \big( Dh(u(t))\, u'(t) \big) \psi_{1,b}. $
Since $u'(t)= w'(t) + \psi_{1,b} \ne 0$, both sides are zero if and only if  $u'(t)$ lies in the kernel of $DF(u(t))$. Thus, the critical points of $h(u(t))$ and of $F$ are the same and, at such point $u(t_0)$,  $\ker DF(u(t_0)) =  \langle u'(t) \rangle$ and $u'(t_0) = c\phi_1(u(t_0))$ for some $c \in \RR$. Now, $\langle \phi_1(u(t)), \psi_{1,b} \rangle > 0$, as both functions are positive, and $\langle u'(t_0) , \psi_{1,b} \rangle = \langle w'(t_0) + \psi_{1,b} , \psi_{1,b} \rangle  = \langle \psi_{1,b} , \psi_{1,b} \rangle >0$, so that $c>0$.

Since $\lambda_1$ and $h'$ have common roots in each fiber, it suffices to show that $p$ is well defined in neighborhoods of these roots. Clearly
\[ \langle DF(u(t))  \phi_1(u(t)), u'(t) \rangle = \lambda_1(u(t)) \langle  \phi_1(u(t)), u'(t) \rangle =  h'(u(t)) \langle \phi_1(u(t)), \psi_{1,b} \rangle . \]
At a critical point $u'= c\, \phi_1(u) ,  c>0$  and nearby $ \langle  \phi_1(u(t)), u'(t) \rangle >0$. Thus $p$, a quotient of smooth nonzero inner products, is smooth.
\qed
\end{proof}

\smallskip
For the reader's convenience, we transcribe the argument in \cite{CTZ1} describing the asymptotic behavior of $F$ along a fiber, already obtained in \cite{BP}.

\begin{prop} \label{prop:infinity}
Let $f$ satisfy  $(1)$ and $(4)$. Then on each fiber $u(t) \in B^2_b$,
   \[ \lim_{t \to \pm \infty} h(u(t)) = - \infty. \]
\end{prop}

\begin{proof} Since $F(u(t)) = z + h(u(t)) \psi_{1,b}$ for $z \in W^0$ and $W^0  \perp V$,
we have $h(u(t) ) = \langle F(u(t)), \psi_{1,b} \rangle$. Since $F(u) =  - \Delta_b \, u - f(u)$ and $ \ - \Delta_b \, \psi_{1,b}= \mu_{1,b} \psi_{1,b}$,
\[ h(u(t) ) =  \mu_{1,b} \, t - \int_{\Omega} f(u(t)) \, \psi_{1,b} . \]
From hypothesis (4), $f(x)$ is bounded below by two lines,
\[ (\mu_{1,b} - \epsilon)\ x + c_- \ , \quad (\mu_{1,b} + \epsilon)\  x - c_+ \ < \ f(x).  \]
We consider $ t \to \infty$. Since $u(t) = w(t) + t \, \psi_{1,b}$, for $w(t) \in W^2 \perp V$,
\[ h(u(t)) \le \mu_{1,b} t - \int_\Omega \big( (\mu_{1,b} + \epsilon) (w(t) + t \psi_{1,b}) - c_+ \big) \, \psi_{1,b} \le - \epsilon t + c_+ \int_\Omega \psi_{1,b} \]
and we are done ($t \to- \infty$ is similar). \qed
\end{proof}

\subsection{ Two-valued potentials} \label{four}

Denote the usual argument from $\RR^2 \setminus \{0\}$ to $[0,2\pi)$ by $\arg$. The {\it sector} $S(\theta)$ is
\[ S(\theta) = \{ x = (x_1, x_2,\ldots,x_n) \in \RR^n , \quad \arg(x_1,x_2) \le \theta \ \hbox{ or } \ (x_1, x_2) = 0 \,  \} \, . \]
We abuse notation slightly and define $S(0) = \emptyset$ and $S(2\pi) = \RR^n$.

For a fixed $p \in \Omega \subset \RR^n$, we consider the translated sector $p + S(\theta)$ and split
\[ \Omega = ( \Omega \cap (p + S(\theta)) \, ) \, \cup \, ( \Omega \cap (p + S(\theta))^c \, ) = \Omega_\theta \cup \Omega_\theta^c  \]
in disjoint subsets with characteristic functions  $\chi_\theta$ and $\chi^c_\theta$ --- the point $p$ stays fixed and is omitted. For $\theta \in (0, 2\pi)$, both sets have nonzero measure. For the proof of Theorem \ref{Rec}, a family of parallel hyperplanes would suffice. In Section \ref{section:cusp2} (more precisely, Proposition \ref{prop:goodnonfold2}) an appropriate choice of $p$ is convenient.

The set of {\it two-valued functions} is
\[ \overline{{\cal V}} = \{\overline  q(L, R, \theta) = L \, \chi_\theta + R \, \chi^c_\theta, \ L, R \in \RR, \ \theta \in [0, 2\pi]  \, \} \ \subset \  L^\infty(\Omega). \]

We simplify  notation: for $\overline q = \overline q(L, R,\theta) \e L \, \chi_\theta + R \, \chi^c_\theta \in \overline{{\cal V}}$, set
\[ T_{\overline q}: H^2 \to H^0, \, v \mapsto - \Delta_b v - \overline q(L, R,\theta) \, v \, , \]
so that the smallest eigenvalue and positive normalized eigenvector restrict to
\[ \lambda_1^{\overline q}: \RR^2 \times [0, 2\pi] \to \RR \quad \hbox{and} \quad \phi_1^{\overline q}: \RR^2 \times [0, 2\pi] \to H^0 \, .\]

A triple  $(L, R, \theta) \in \RR^2 \times [0, 2\pi] $ is {\it balanced} if $\lambda_1^{\overline{q}} (L, R,\theta) = 0$. The next lemma shows that balancing is frequently feasible.

\medskip

\begin{lemma} \label{lemma:balance} There is a  continuous function
\begin{align*}
\Theta: \big( (-\infty, \mu_{1,b}] \times (\mu_{1,b}, \infty) \big) \cup
\big( (\mu_{1,b}, \infty) \times (-\infty, \mu_{1,b}] \big)  &\to (0, 2\pi) \\ (L, R) &\mapsto \,  \Theta(L, R)
\end{align*}
such that $\lambda_1^{\overline q}(L, R, \Theta(L, R)) = 0$. Also, $\Theta(\mu_{1,b},R) = 2\pi$ and $\Theta(L,\mu_{1,b}) = 0$.
\end{lemma}

\smallskip
\begin{proof} Clearly, the map
$\iota: (L, R,\theta) \in \RR^2 \times [0, 2\pi] \mapsto  \overline q(L, R,\theta) \in \Hzb$ is continuous, so that $ \lambda_1^{\overline q}: \RR^2 \times [0, 2\pi] \to \RR$ and $\phi_1^{\overline q}: \RR^2 \times [0, 2\pi] \to H^0$ are too:
proceed as in  the proof of Proposition  \ref{prop:fantasma} with the composition $\Phi \circ \iota$.

Suppose that $L \le \mu_{1,b} < R$, the other case is similar.
We have
\[ \lambda_1^{\overline q}(L, R,0) = \mu_{1,b} - R < 0 \quad \hbox{ and } \quad \lambda_1^{\overline q} (L, R,2\pi) = \mu_{1,b} - L > 0 \, , \, \hbox{for} \, L \ne  \mu_{1,b} \, , \]
\[ \lambda_1^{\overline q} (\mu_{1,b},R,2\pi) =  0 \quad \hbox{ and } \quad \lambda_1^{\overline q} (\mu_{1,b},R,\theta) \le 0 \, , \, \hbox{for} \, \theta \in [0,2\pi) \, . \]
Thus, for a given $L \in (-\infty, \mu_{1,b}]$, a balancing $\Theta$ exists by continuity in $\theta$ of $\lambda_1^{\overline q}$. We now show uniqueness and continuity in $L$.

For $L_1 \le L_2$, $R_1 \le R_2$ and $0< \theta_1 \le \theta_2 < 2 \pi$, consider points $(L_1,R_1,\theta_1)$ and $(L_2,R_2,\theta_2)$ associated with potentials $\overline q_i$, operators $T_{\overline q_i}$ and  quadratic forms $Q_i(v) = \langle v, T_{\overline q_i} v \rangle$, for $i=1,2$. Clearly, $\overline q_1 \le \overline q_2$ pointwise a.e.,  $\lambda_1^{\overline q_1} \ge \lambda_1^{\overline q_2}$,   $\lambda_1^{\overline q_i} = Q_i(\phi_1^{\overline q_i})$ and $\phi_1^{\overline q_i} >0$ in $\Omega$. Also, $T_{\overline q_1} = T_{\overline q_2} + q_+$, for a potential $q_+ \ge 0$. If $(L_1,R_1,\theta_1)$ and $(L_2,R_2,\theta_2)$ are distinct, $q_+ \not \equiv 0$ and
\[ \lambda_1^{\overline q_1} = Q_1(\phi_1^{\overline q_1}) = Q_2(\phi_1^{\overline q_1}) + \langle q_+ \, \phi_1^{\overline q_1}, \phi_1^{\overline q_1} \rangle  > Q_2(\phi_1^{\overline q_1}) \ge \lambda_1^{\overline q_2} \, .\]
Thus $\lambda_1^{\overline q}$ is strictly monotonic on each coordinate and $\Theta(L, R)$ is well defined for $(L, R,\theta) \in  (-\infty, \mu_{1,b})\times (\mu_{1,b},\infty) \times (0, 2\pi)$. Similarly, $\Theta$ is strict monotonic along the segment $(\mu_{1,b},R,\theta)  \, , \, \theta \in [0,2\pi]$, so that  $\lambda_1^{\overline q} (\mu_{1,b},R,\theta) < 0 \, , \, \theta \in [0,2\pi)$, enforcing $\Theta(\mu_{1,b},R) = 2\pi$. Finally, from the continuity of $\lambda_1^{\overline q}$, $\Theta$  is continuous. \qed
\end{proof}

\medskip
Throughout the text, we deal with two-valued functions $\overline u(\ell,r) = \ell  \, \chi_\theta + r \, \chi^c_\theta$, giving rise to potentials
\[ q(L,R) = f'(\overline u(\ell,r)) = f'(\ell)  \, \chi_\theta + f'(r) \, \chi^c_\theta = L  \, \chi_\theta + R \, \chi^c_\theta \, . \]
We use capital and lower letters to specify the different nature of the quantities. Abusing notation slightly, we say in this case that $(\ell,r,\theta)$ is balanced if the potential associated with $\big(\, f'(\ell),\, f'(r), \, \theta \, \big)$ is.

\subsection{A point with four preimages  } \label{subsection:psi}

As stated in the sketch of proof above, we search for a critical point $u(t_0)$ for which the height $h$ along its fiber $u(t)$ has a strict local minimum.
To classify the local extremum we use an additional derivative, by slightly amplifying Proposition \ref{prop:fantasma}.

\begin{prop}\label{prop:derlambda} Let $f$ satisfy hypothesis $(1)$. The
derivative of the smooth function $\lambda_1: B^2_b\to \RR$ along $v \in B^2_b$ is
\[ D \lambda_1(u) \ v =  \langle \nabla \lambda_1(u) , v \rangle= - \int_\Omega f''(u) \phi_1^2(u)\ v \, , \quad \nabla \lambda_1(u) = - f''(u) \phi_1^2(u) \in \Hzb \,   .\]
The extension $\lambda_1: \Hzb \to \RR$ admits Gateaux derivatives along  $v \in \Hzb$ given by the same formulas.
The map $\nabla \lambda_1: \Hzb \to \Hzb, \ u \mapsto \nabla \lambda_1(u)$ is $ub$-continuous.
\end{prop}

For the proof see Section \ref{subsection:spectrumu}. Define the $ub$-continuous map
\[ \delta_1: \Hzb \to \RR \,  ,\quad u \, \mapsto \,    D \lambda_1(u) \phi_1(u) \, = \, - \langle \ f ''(u) \ \phi_1(u)^2 , \ \phi_1(u)  \ \rangle \, . \]
Say $r \sim s$ if both numbers $r$ and $s$ have the same sign.
We  search for a function $u(t_0)$ for which, from Proposition \ref{prop:hcrit},
\[ \frac{d}{dt} \, h(u(t))|_{t= t_0} = p(u(t_0)) \ \lambda_1(u(t_0)) \, = \, 0 \, ,\]
\[ \frac{d^2}{dt^2} \, h (u(t))|_{t= t_0} \, = \, D p(u(t_0)) u'(t_0) \ \lambda_1(u(t_0)) + p(u(t_0)) \ D \lambda_1(u(t_0)) u'(t_0) \]\[
\, \sim  D\lambda_1(u(t_0)) u'(t_0) \, \sim \, \delta_1(u(t_0)) \, > \, 0 \, . \]

For a given  standard boundary condition, let $D \subset B^2_b$ be a dense subset of $\Hzb$.
In Propositions \ref{prop:psipos} and \ref{prop: molly}, we find respectively a two-valued function $\overline u \in H^0$ and then $u \in D$ which are zeros of $ \lambda_1$  and on which $ \delta_1$  takes positive values.

\medskip

The sequences $\{ x_m ^{ \pm}\}$ in the next lemma play the role of almost critical points.

\begin{lemma} \label{lemma:PS} Let $g: \RR \to \RR$ be a smooth Morse function with bounded image containing an interior point $\mu$.  Then there are sequences $x_m^+, x_m^- \in \RR$  satisfying the following properties.
\begin{enumerate}
\item $\lim_{m \to \infty}  g' (x_m^+) = \lim_{m \to \infty} g'(x_m^-) = 0 \, , \quad  g'(x_m^+), \  g'(x_m^-) \ne 0 \,  ,$
\item $\lim_{m \to \infty} g(x_m^+) = R^+ > \mu \,  , \ \lim_{m \to \infty} g(x_m^-) = R^- < \mu \, .$
\end{enumerate}

\end{lemma}

We will  use the lemma above for the function $g = f ' $ and $\mu= \mu_{1,b}$.

\smallskip

%

\begin{proof}  Without loss, $\mu=0$: the general case follows by translation.  We prove the existence of $x^+_m$, the other case being similar. Since $\mu=0$ is an interior point of $g(\mathbb R)$,  $A^+=\{x\in\mathbb R: \ g(x)>0\}$ is nonempty.  If $A^+$ contains a critical point $x^+$, then $g(x^+)>0$,
$g'(x^+)=0$,  and since $g$ is a Morse function, the existence of the requested sequence $x^+_m\to x^+$   is clear.

Suppose that $A^+ \ne \RR$ has no critical points: by the monotonicity of $g$ in its connected components,  $A^+$ necessarily contains a semi-infinite interval, say $(c,+\infty) \ne \RR$ (the case $(-\infty,c)$ being similar), in which $g'$ is strictly positive. Since $g$ is  bounded, as $x\to +\infty$ we have $g \to R^+ > 0$. For the existence of the desired $x^+_m$, we show that
\[ \forall \epsilon >0 \, , \ \forall N \in \RR \, , \ \exists x  \quad x > N \  \hbox{ and } \ g'(x) < \epsilon \, .\]
Arguing by contradiction, there are $\epsilon >0$ and $N \in \RR$ for which $g'(x) \ge \epsilon$ if $x>N$. But then $g$ is unbounded: integrate $g'$ on intervals  $[N, N+p], p >0$.
\qed
\end{proof}

%

\smallskip
\begin{prop}\label{prop:psipos}   Suppose $(1), (2), (3)$. There is $\overline u  \in \overline{{\cal V}}$ with $ \lambda_1(\overline u ) = 0 , \delta_1(\overline u)>0$.
\end{prop}

\begin{proof}  From  $(3)$, take  $x_\ast$ with $f''(x_\ast) = 0$ and thus $f'(x_\ast) \ne \mu_{1,b}, \, f'''(x_\ast) \ne 0$, by $(2)$. Suppose also
$f'(x_\ast) < \mu_{1,b}$  --- the other case is similar.
By $(2)$, there is a point $x_\ast^-$ arbitrarily close to $x_\ast$ for which $f''(x_\ast^-) <0 $, with $f'(x_\ast^-) < \mu_{1,b}$.

For $p \in \Omega$, we consider translated sectors $p + S(\theta)$ defined in Section \ref{four}.
For $x_m^+$ obtained from Lemma \ref{lemma:PS} with $g = f '$, set
\[ \overline u_m^+ = \overline u( x_\ast^-,\,  x_m^+, \, \Theta_m = \Theta(x_\ast^-,x_m^+)) = x_\ast^- \, \chi_{\Theta_m} + x_m^+ \, \chi^c_{\Theta_m} \ , \]
so that $\lambda_1(\overline u_m^+) =0$ by balancing (Lemma \ref{lemma:balance}): notice that $f'(x_m^+) \to R^+ > \mu_{1,b}$.
From the definition of a standard boundary condition, the eigenfunctions $\phi_1(\overline u_m^+)$ of the potentials $\overline q(f'(x_\ast^-), f'(x_m^+), \Theta_m) =  f'(\overline u_m^+)$, are uniformly bounded and converge to the eigenfunction $\phi_1^\infty$ for the potential $\overline q(f'(x_\ast^-), R^+, \Theta_\infty = \Theta(f'(x_\ast^-), R^+))$. By the continuity of $\delta_1$ (Proposition \ref{prop:derlambda}),
\[ \lim_{m \to \infty}  \delta_1(\overline{u}_m^+) =  - \int_{\Omega_{\theta_\infty}} f'' (x_\ast^-) (\phi_1^\infty)^3  - \int_{\Omega^c_{\theta_\infty}}\big(\lim_{m \to \infty}  f'' (x_m^+) \big) (\phi_1^\infty)^3 > 0 \  \]
and thus, for some large $K$, $\delta_1(\overline{u}_{K}^+) > 0$: take $\overline u = \overline{u}_{K}^+$.
\qed
\end{proof}

\medskip
We now mollify $\overline{u}_K$. Recall that $D \subset B^2_b$ is dense in $\Hzb$.

\medskip
\begin{prop} \label{prop: molly} Assume $(1),(2),(3)$. There is  $u\in D$ with $\lambda_1(u) = 0,\ \delta_1(u) > 0$.
\end{prop}

\begin{proof} For $\overline u_K= \overline u(x_\ast^-,x_K^+,\Theta_K) \in \Hzb$ as above, $\lambda_1(\overline u_K)=0, \delta_1(\overline u_K) >0$. Let $U \subset \Hzb$ be an open ball centered in $\overline u_K$ in which $\delta_1$ is positive. Consider the segment $\overline u(t)= \overline u(x_\ast^- + t,x_K^+,\Theta_K) \in H^0$ for $t$ near zero (notice that $\Theta_K$ stays fixed). Since $f''(x_\ast^-) <0$, from Proposition \ref{prop:derlambda}, $\lambda_1(\overline u(t))$ at $t=0$ is strictly monotonic. Thus, there are  $\overline u^+, \overline u^- \in U$ with $\lambda_1(\overline u^+) >0$ and $\lambda_1(\overline u^-)< 0$. Now take  $u^+,  u^- \in U \cap D$ for which $\lambda_1(u^+) >0$ and $\lambda_1( u^-)< 0$. By continuity, there is $u \in D$ in the segment joining $u^+$ and $ u^-$ for which $\lambda_1(u) = 0$ and  $\delta_1(u) > 0$.
\qed
\end{proof}
\begin{figure}[ht]
\begin{center}
\epsfig{height=45mm,file=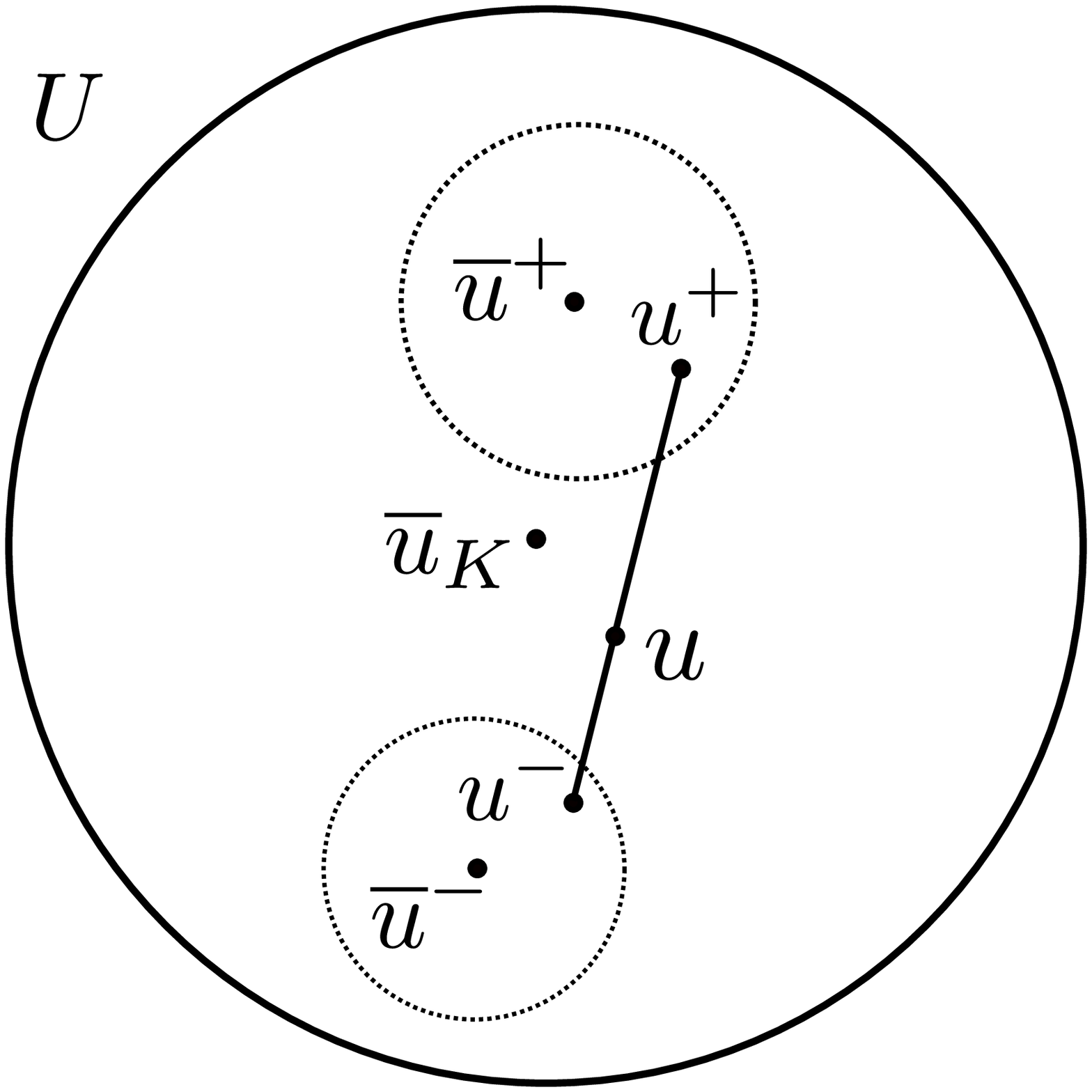}
\end{center}
\end{figure}

\bigskip
This completes the proof of Theorem  \ref{Rec}: for the function $u$ above, $F(u)$ has (at least) four preimages. Notice that a weaker version of hypothesis $(2)$ would suffice: the fact that $f' - \mu_{1,b}$ and $f'''$ have no common roots is easily removed.

\section{Cusps and Theorem \ref{theo:c}} \label{section:cusp1}

The proof of Theorem \ref{theo:c}  requires a few steps. In
Section \ref{section:defcusp} we validate the characterization of a cusp  as a zero $u_{nf}$ of the function $\Lambda$, together with the
transversality condition $\tau_1(u_{nf}) \ne 0$. The next step is to obtain a zero $\overline u_{nf}
\in \overline{{\cal V}}$ (Section \ref{section:nonfold1}) and then by mollification a
smooth zero $u_{nf}$ (Section \ref{section:mollifiers}) for which $\nabla \lambda_1(u_{nf})$ and $\nabla \delta_1(u_{nf})$ are independent. Finally we show that some $u_c$ nearby does everything we want: it is either a cusp (i.e., it satisfies $\tau_1(u_{nf}) \ne 0$) or a point for which $F(u_c)$ has a full interval of preimages (Section \ref{section:transversality}).

\subsection{Folds, nonfolds and cusps} \label{section:defcusp}

From the arguments of Section \ref{section:theo3}, there are
fibers on which the height $h$ has  local maxima and minima. A
coalescence of both extrema would yield  a critical point of $h$
for which the {\it second} derivative along a fiber is zero
--- a point like this is a {\it nonfold}, which we proceed to
describe in detail. Nonfolds which satisfy additional generic
properties are {\it cusps}.

Nonfolds and cusps are identified with local checks: they do not require asymptotic hypotheses like $(4)$. For maps between spaces of finite dimensions, folds and cusps are special cases of {\it Morin singularities} (\cite{Mo}), which generically correspond to critical points $u$ at which the differential $DF(u)$ has kernel of dimension equal to one. Their characterization is amenable to extensions: for example, cusps of functions between Banach spaces are well known
(\cite{CDT},\cite{ChT2},\cite{LM},\cite{R1},\cite{R2},\cite{MST}). We are especially interested in the geometric characterization in p. 183 of \cite{ChT2}, which we transcribe.

For $X$ and $Y$ real Banach spaces, let $U \subset X$ be an open set, $G: U \to Y$ a smooth function. Recall that $\CCC \subset U$, the critical set of $G$, is the set of points of $U$ on which $DG$ is not invertible. We require $\CCC$ to be a manifold. Also, let ${\cal N} \subset {\cal C}$ consists of critical points $u$ for which $ \Ker DG(u) \subset T_{u} \, {\cal C}$.
Then $u_c \in U$ is a {\it cusp} of $G$ if and only if
\begin{enumerate}[(a)]
\item $DG(u_c)$ is a Fredholm map of index $0$ and $\dim \Ker DG(u_c) = 1$,
\item $\CCC \subset U$ is a manifold of codimension 1  and $ u_c \in {\cal N}$, i.e., $ \Ker DG(u_c) \subset T_{u_c} \, {\cal C}$,
\item ${\cal N} \subset \CCC$ is a manifold of codimension 1  of $\CCC$
and $ \Ker DG(u_c) \not \subset T_{u_c} \, N$.
\end{enumerate}

First, we settle the smoothness of the functions in Theorem \ref{theo:c}. Some of the required computations are known (\cite{R1} is a good example). From hypothesis $(1)$, for  $u \in \Hzb$, the eigenpair $\lambda_1(u), \phi_1(u)$ is well defined. Let $V = \langle \phi_1(u) \rangle$, and  consider the orthogonal split $\Hzb \e W \oplus V$, $H^2_b = W^2 \oplus V$ with associated projection $\Pi_W: \Hzb \to W$. Denote the restriction of an operator $T$ to $W$ by $T_W$.

\begin{prop} \label{prop:Lambda} Suppose hypothesis $(1)$.
Let $f:\RR \to \RR$ be smooth.  The function
\[ \Lambda: B^2_b \to \RR^2 \, , \quad \Lambda(u) \, = \,  (\, \lambda_1(u) \, , \ \delta_1(u) \,  )  \]
is smooth. The gradient of $\delta_1: B^2_b \to \RR$ is \[ \nabla \delta_1(u) = - f'''(u) \phi_1^3(u) - 3 w(u) f''(u) \phi_1(u) \in \Hzb \, , \]
where $w(u) = D\phi_1(u) \, \phi_1(u) = \big( DF(u) - \lambda_1(u) I \big)_W^{-1} \Pi_W(f''(u) \phi_1^2(u)) \in W^2 \subset H^2$.
The  extension $\Lambda: \Hzb \to \RR^2$, the gradient  $\nabla \delta_1 : \Hzb \to \Hzb$ and the functional
\[ \tau_1: \Hzb \to \RR \, , \quad \tau_1(u) \e D\delta_1(u) \,  \phi_1(u) \e
\langle \, \nabla \delta_1 (u) \, , \, \phi_1(u) \, \rangle \]
are $ub$-continuous.
For any  $2$-dimensional affine subspace $V_\ast \subset \Hzb$, the restriction $\Lambda|_{V_\ast}:V_\ast \to \RR^2$ is a $\CCC^1$ map.
\end{prop}

The proof is given in Section \ref{section: Lambda}.
We now show that the requirements in Theorem \ref{theo:c} indeed yield a cusp in the sense above.

\bigskip
\begin{prop} \label{prop:taubasta} Assume standard boundary conditions and hypotheses $(1)$ and $(2)$. A zero $u_c$ of $\Lambda: B^2_b \to \RR^2$ for which $\tau_1(u_c) \ne 0$ is a bona fide cusp of $F$.
\end{prop}
\begin{proof}
For standard boundary conditions and hypothesis $(1)$, $\lambda_1$ and  $\phi_1$ are globally defined, even for $u \in H^0$. In particular, $(a)$ is satisfied. Denote  the levels $\lambda_1^{-1}(0)$ and $\delta_1^{-1}(0)$ by $\CCC$ and $\DDD$, so that $u_c \in {\cal N} = \CCC \cap  \DDD \subset B^2_b$. Most of the result follows from the linear independence of the gradients of $\lambda_1$ and $\delta_1$
at points $u \in {\cal N}$: then near ${\cal N}$ the sets $\CCC$ and $\DDD$ are hypersurfaces and ${\cal N}$ is a manifold, obtained by their transversal intersection.

First,  $\nabla \lambda_1(u_c) = - f''(u_c) (\phi_1(u_c))^2 \not \equiv 0$. Indeed, since the zeros of $f''$ are isolated and $u_c \in B^2_b$ is continuous, if $f''(u_c) \equiv 0$ then  $u_c$ must be constant, implying $\lambda_1(u_c) = 0 $ and $f ' (u_c) = \mu_{1,b}$, which is excluded by  $(2)$.

Thus $\nabla \lambda_1$ and $\nabla \delta_1$ are dependent only if $c  \nabla \lambda_1(  u_c) = \nabla \delta_1 (  u_c)$, for some $c \in \RR$. From the expression for $\nabla \delta_1$ in Proposition \ref{prop:Lambda},
\[  c f'' (  u_c) \phi_1^2(  u_c) = - f'''(  u_c) \phi^3_1(  u_c) - 3 w f''(  u_c) \phi_1(  u_c)  , \]
so that $(c \phi_1 + 3 w)f'' = - f'''\phi_1^2$. In  $\Omega$, $\phi_1 >0$, and $u_c \in {\cal N}$ implies \[ \delta_1(u_c) = \int_\Omega f''(u_c) \phi_1^3(u_c) = 0 \, , \]
so that $f''(u_c)$ changes sign. Thus the zeros of $f''(u_c)$ must be zeros of $f'''(u_c)$, again contradicting hypothesis $(2)$: the gradients are independent.  Also, since $\ker DF(u_0)$ is spanned by $\phi_1(u_0)$, (b) is satisfied, together with  the first part of $(c)$. The second part of $(c)$ is exactly the requirement $\tau_1(u_c) \ne 0$. \qed
\end{proof}

%

\bigskip
We  make some distinctions among critical points.
A {\it fold} $u_f$ is a regular critical point of $F$ for which $\phi_1(u_f) \notin T_{u_f}{\cal C}$ (i.e. $\delta_1(u_f) \ne 0$),
 the generic situation.
A {\it nonfold} $u_{nf} \in \Hzb$ is a zero of $\Lambda: \Hzb \to \RR^2$ and it is {\it regular} if $D\Lambda(u_{nf})$ is surjective. Thus, a cusp $u_c \in B^2_b$ is a regular nonfold for which $\tau_1(u_c) \ne 0 $.

\subsection{A regular nonfold $\overline u_{nf}$} \label{section:nonfold1}

The next result is the key ingredient in the proof of Theorem \ref{theo:c}. In a sense, it plays the role of balancing for the functional $\delta_1$.

\begin{prop} \label{prop:goodnonfold1} Suppose $(1),(2)$ and $(3)$ and standard boundary conditions.  Then there is a regular nonfold $\overline u_{nf} \in \overline{{\cal V}}$.
\end{prop}

\begin{proof} We say that two points $x, y \in \RR$ are {\it opposite} with respect to a continuous function $h: \RR \to \RR$ if $h(x) h(y) <0$. In order to $\overline u_{nf} = \ell \, \chi_\theta + r \, \chi^c_\theta$ satisfy $ \lambda_1(\overline u_{nf}) = 0$, either $f'(\ell) = f'(r) = \mu_{1,b}$ or
$\ell$ and $r$ are opposite with respect to $f' - \mu_{1,b}$. Conversely, $f'(\ell) = f'(r) = \mu_{1,b}$ implies $\lambda_1(\overline u_{nf})=0$ for any $\theta \in [0,2\pi]$, and otherwise a unique balancing $\Theta(\ell,r)$ provided by Lemma \ref{lemma:balance} obtains $\lambda_1(\overline u_{nf})=0$.

Similarly,  $\delta_1(\overline u_{nf})=0$ requires that either $f''(\ell)=f''(r)=0$ or $\ell$ and $r$ are opposite with respect to $f''$. In the construction that follows, we also take $\ell$ so that $f'''(\ell) = 0$, to be used in the proof that $\overline u_{nf}$ is a regular nonfold.

We must consider a few alternatives. From hypothesis $(3)$, there is $x_\ast$ with $f''(x_\ast)=0$.  By hypothesis $(2)$, $f'(x_\ast) \neq \mu_{1,b}$ and $f'''(x_\ast) \ne 0$.
Suppose without loss that $f'(x_\ast) < \mu_{1,b}$. Moreover we can suppose $f'''(x_\ast)>0$. Indeed, if
$$
f'(x_\ast)<\mu_1, \ f''(x_\ast)=0, \ f'''(x_\ast)<0
$$
there will be another stationary point $x_{\ast\ast}$ for $f'$ such that

$$
f'(x_{\ast\ast})<\mu_1, \ f''(x_{\ast\ast})=0, \ f'''(x_{\ast\ast})>0
$$
in virtue of hypothesis (1), and we take $x_\ast = x_{\ast\ast}$. Let $I_{-}=(m_-,x_\ast)$ and $I_+=(x_\ast,m_+)$ be the maximal monotonic intervals of $f'$ (i.e., the largest intervals on which $f'' \ne 0$) having $x_\ast$ as an endpoint. By hypotheses $(1)$ and $(2)$, we may suppose that one such interval crosses $\mu_{1,b}$.

Suppose  $f'$  stays below $\mu_{1,b}$ on $I_-$ but crosses $\mu_{1,b}$ for $y_\mu \in I_+$: $f'(y_\mu) = \mu_{1,b}$. Take $\ell$ to be the point with $f'''$ equal to zero in $I_-$: the interval of opposite points in which to find the appropriate $r$ is $I_\mu = (y_\mu, z_\mu)$, the  part of $I_+$ above $\mu_{1,b}$. The argument is the same if the roles of the intervals are interchanged.

Suppose first that $z_\mu < \infty$, and thus $f''(z_{\mu}) =0$, so that $\ell$ is opposite to all points in $I_\mu$ with respect both to $f' - \mu_{1,b}$ and $f''$.

We compute $\delta_1$ at extreme points of $I_\mu$. For $\overline u = \overline u(\ell, y_\mu, \Theta(\ell,y_\mu))$, we have \[ \Theta(\ell,y_{\mu})= 0 \, , \quad \mu(\Omega_{0}) =0 \, , \quad f''(y_\mu) > 0 \,  ,\]
\[ \delta_1(\overline u(\ell, y_\mu, 0))
 =  - \int_{\Omega_{0}} f'' (\ell) (\phi_1(\overline u))^3  - \int_{\Omega^c_{0}} f'' (y_\mu)  (\phi_1(\overline u))^3 < 0  \, .  \]
For $\overline u = \overline u(\ell, z_\mu, \Theta(x_{\mu}, z_\mu))$, since $f''(\ell) <0$ and  $f''(z_{\mu}) = 0$,
 \[ \delta_1(\overline u(\ell, z_\mu, \Theta(x_{\mu}, z_\mu)))
 =  - \int_{\Omega_{\Theta(\ell, z_{\mu})}} f'' (\ell) (\phi_1(\overline u))^3  - \int_{\Omega^c_{\Theta(\ell, z_{\mu})}} f'' (z_{\mu})  (\phi_1(\overline u))^3 > 0  \, .  \]

By the continuity of $\delta_1$ (Proposition \ref{prop:derlambda}), there is $r \in I_\mu$ for which $\overline u_{nf} = \overline u(\ell, r, \Theta(\ell,r))$ is a common zero of $\lambda_1$ and $\delta_1$.

Suppose now $z_\mu = \infty$. Take $z_m^+ \to \infty$ as in Lemma \ref{lemma:PS} for $g = f'$, so that $f ' (z_m^+) > \mu_{1,b} $ and $f '' (z_m^+) \to 0$, and follow the proof of Proposition \ref{prop:psipos}. The case when $f'$ is strictly greater than $\mu_{1,b}$ close to $y_\mu$ on the left is similar.

\bigskip

If $f'$  crosses $\mu_{1,b}$ both on $I_-$ and $I_+$, take $\ell \in I_-$ with $f'''(\ell) = 0$ and search for $r$ in the  opposite subinterval in $I_+$.

\bigskip
We now show that $\overline u_{nf}$ with $f'''(\ell)=0$ is a {\it regular} nonfold. From hypothesis $(2)$, $f''(\ell) \ne 0 $ and  $\nabla \lambda_1(\overline u_{nf}) = f''(\overline u_{nf}) \phi_1^2(\overline u_{nf}) \ne 0$. Thus the gradients of $\lambda_1$ and $\delta_1$ are dependent if and only if $c . \nabla \lambda_1(\overline u_{nf}) = \nabla \delta_1 (\overline u_{nf})$, for $c \in \RR$. From the expression for $\nabla \delta_1$ in Proposition \ref{prop:Lambda},
\[  c f'' (\overline u_{nf}) \phi_1^2(\overline u_{nf}) = - f'''(\overline u_{nf}) \phi^3_1(\overline u_{nf}) - 3 w f''(\overline u_{nf}) \phi_1(\overline u_{nf})  , \]
so that $(c \phi_1 + 3 w)f'' = - f'''\phi_1^2$.  On $\Omega_{\Theta(\ell,r)}$, where  $f'''(\ell)=0$ (and then $f''(\ell) \ne 0$),
\[ w(\overline u_{nf}) = c' \phi_1(\overline u_{nf}) \, , \quad \hbox{for some } \ c' \in \RR \, . \]
From the  expression for $w$ in Proposition \ref{prop:Lambda}, taking into account that $\lambda_1(\overline u_{nf})=0$ and $\delta_1(\overline u_{nf})=0$, so that $\Pi_W(f''(u) \phi_1^2(u)) = f''(u) \phi_1^2(u)$,
\[ w(\overline u_{nf}) = \big( DF(\overline u_{nf}) )_W^{-1} \, (f''(\overline u_{nf}) \phi_1^2(\overline u_{nf})) .\]
Since $DF(\overline u_{nf})$ is a local operator, one may apply it on both sides of
\[ w(\overline u_{nf})(x) = c ' \phi(\overline u_{nf})(x) \, , \quad x \in \Omega_{\Theta(\ell,r)} \ne \emptyset \, , \]
to conclude that $  f''(\ell) \phi_1^2(\overline u_{nf})(x) = 0 \, , \, x \in \Omega_{\Theta(\ell,r)}, $
clearly a contradiction. \qed
\end{proof}

\subsection{Smoothing: from $\overline u_{nf}$ to a regular  nonfold $u_{nf} \in D$} \label{section:mollifiers}

Let $D \subset B^2_b$ be a dense subspace of $\Hzb$.
We obtain a regular nonfold $u_{nf} \in D$  by mollification of $z_\ast=\overline u_{nf}$ from Proposition \ref{prop:goodnonfold1}. Consider a ball $B_{z_\ast}(r) \subset \Hzb$ in the sup norm. By Proposition \ref{prop:derlambda}, $\Lambda: B_{z_\ast}(r) \subset \Hzb \to \RR^2$ is continuous  ($ub$-continuous functions on $L^\infty$-bounded sets are continuous) and $\Lambda(z_\ast) = 0 $.

The existence of two functions $v_1, v_2 \in  D$ with invertible Jacobian
\[ \left(\begin{array}{rr}
\langle \, \nabla \lambda_1(z_\ast) \, , \, v_1 \, \rangle &
\langle \, \nabla \lambda_1(z_\ast) \, , \, v_2 \, \rangle\\
\langle \, \nabla \delta_1(z_\ast) \, , \, v_1 \, \rangle &
\langle \, \nabla \delta_1(z_\ast) \, , \, v_2 \, \rangle\\
\end{array}\right) \]
is clear, since $z_\ast$ is a regular nonfold. Let $\tilde V_\ast \subset D$ be the span of $v_1$ and $v_2$ and set $V_\ast = z_\ast + \tilde V_\ast$.
Thus $z_\ast \in \Hzb$ is a regular point of the  restriction $\Lambda_{\ast} :V_\ast \to  \RR^2$, which is a $\CCC^1$ map from Proposition \ref{prop:Lambda}. Thus, for small balls $B_{z_\ast}(\epsilon) \subset B_{z_\ast}(r)$, the topological degree satisfies $\deg(\Lambda_{\ast}, B_{z_\ast}(\epsilon) \cap V_\ast, 0) = \pm 1$ .

Take $z_m \in D \cap B_{z_\ast}(\epsilon)$ with $z_m \ub z_\ast$.   Define $V_m = z_m + \tilde V_\ast \subset D$. For large $m$, the restrictions $\Lambda_{\ast} : B_{z_\ast}(\epsilon) \cap V_\ast \to \RR^2$ and $\Lambda_{m} : B_{z_m}(\epsilon) \cap V_m \to \RR^2$ get arbitrarily close in the uniform norm after composing with the obvious translation. For a small ball $B_{z_\ast}(\epsilon)$,
\[ \deg(\Lambda_{\ast}, B_{z_\ast}(\epsilon) \cap V_\ast, 0)  = \deg(\Lambda_{m}, B_{z_\ast}(\epsilon) \cap V_m, 0) \ne 0 \,  .\]

Thus $\Lambda_{m}$ has a zero in $B_{z_\ast}(\epsilon) \cap V_m$ and we obtain a sequence of nonfolds $u_m \in V_m \subset D $ convergent to $z_\ast=\overline u_{nf}$ in $\Hzb$.

The restrictions $\Lambda_{m}: V_m \to \RR^2$ are smooth and their Jacobians at $u_m$ converge to $D\Lambda_{\ast}(z_\ast)$, by Proposition \ref{prop:Lambda}. For large $m$, $D\Lambda(u_m)$ is then surjective and any such $u_m$ is a regular nonfold $u_{nf} \in D$.

\subsection{Regular nonfolds imply cusps or worse} \label{section:transversality}

There is only one thing missing to complete the proof of Theorem  \ref{theo:c}: it is not clear that the regular, smooth nonfold $u_{nf} \in D \subset B^2_b$ satisfies the transversality condition $\tau_1(u_{nf}) = D\delta_1(u_{nf}) \ \phi_1(u_{nf}) \ne 0 $ stated in Theorem \ref{theo:c}.
We show that either the condition is satisfied on an open dense set of regular nonfolds near $u_0$, obtaining then the required cusp $u_c$, or something very implausible happens: an abundance of fibers in the domain, each taken by $F$ to a point.
\bigskip

Recall the zero levels $\CCC, \DDD \in B^2_b$  of $\lambda_1$ and $\delta_1$. From Section \ref{section:mollifiers},  the set of nonfolds ${\cal N} = \CCC \cap \DDD$ in nonempty. From the proof of Proposition \ref{prop:taubasta}, a point $u_{nf} \in {\cal N}$ for which $\tau_1(u_{nf}) \ne 0$ is automatically a cusp. Thus,

\smallskip
\centerline{ {\it $F$ has no cusps if and only if  $\tau_1$ is identically zero in ${\cal N}$.} }
\smallskip
Recall from Section \ref{section:fibers} that a fiber is the inverse under $F$ of a line parallel to the free eigenfunction $\psi_{1,b}$.

%

\bigskip

\begin{prop} \label{Nfibers} Suppose $\tau_1 \equiv 0$ in ${\cal N} \ne \emptyset$ and the hypotheses of Theorem \ref{theo:c}. Then ${\cal N}$ is foliated by fibers, each being sent to a single point by $F$.
\end{prop}
\begin{proof} As in the proof of Proposition \ref{prop:taubasta}, ${\cal N}$ is the transversal intersection of the sets $\CCC$ and $\DDD$, which are manifolds near ${\cal N}$, since the gradients $\nabla \lambda_1(u)$ and $\nabla \delta_1(u)$ are linearly independent.
The differential equation in $B^2_b$
\[ u ' \, = \,  \tilde \phi_1(u) \, = \, \frac{\phi_1(u)}{\| \phi_1(u) \|_{B^2}} \, , \quad u(0) = u_0 \]
has a globally defined solution $\gamma= \{ u(t), t \in \RR\}$. Also, the vector field $u \mapsto \tilde \phi_1(u)$ restricts to a vector field tangent to ${\cal N}$. Indeed, $\delta_1(u) = 0$ and $\tau_1(u)=0$ imply that $\phi_1(u)$ is orthogonal to $\nabla \lambda_1(u)$ and $\nabla \delta_1(u)$ respectively, so that $\phi(u) \in TN_u$. Thus if $u(0) \in {\cal N}$ then $\gamma \subset N$ and $\gamma$ consists only of critical points.

From Proposition \ref{prop:hcrit}, at a critical point $u(t_c)$ the  kernel vector $\phi_1(u(t_c))$ of $DF(u(t_c))$ is the tangent vector to the fiber at $u(t_c)$. Thus integration of the vector field above also yields the fibers through points in ${\cal N}$: ${\cal N}$ is indeed foliated by fibers.
Finally, from Proposition \ref{prop:hcrit}, along fibers $u(t) \in {\cal N}$,
 \[ \frac{d}{dt} \, h(u(t)) \e p(u(t)) \ \lambda_1(u(t)) \e 0 \, . \]
The height $h(u(t))$ does not change: $F$ takes each fiber  to a single point.
\qed
\end{proof}

\bigskip
The proof of Theorem \ref{theo:c}  is now complete, up to the verification in Proposition \ref{prop:noworse} that under certain additional hypotheses we are sure to find cusps for $F$.

\bigskip
Are there functions as suggested by the alternative in the theorem? At a nonfold $u_{nf}$, $F$ would take the (local) normal form near  the origin:  \[ \Psi(w, x, y, z) = (w, \ x,\ y, x \ \alpha(w,x,y,z) + y \ \beta(w,x,y,z)), \, \quad w \in W,\  (x,y,z) \in \RR^3 \, ,\]
where $W$ is a Banach space.
The vertical axes $(w,0,0,z)$ are taken to $(w,0,0,0)$: they all collapse under $F$.  The simplicity of the example is misleading: its rarity is due to the fact that the collapse happens for an open set of $W$. What is not clear is that $\Psi$ indeed is a local form of some function $F$ near some $u_{nf}$. There are nonlinearities for which a few fibers have locally constant height functions \cite{TZ}), but they are far from being as abundant as in the situation above.

\bigskip

\begin{prop}\label{prop:noworse} Consider the hypotheses of Theorem \ref{theo:c}. Each of the two possibilities implies the existence of a cusp of $F$.
\begin{enumerate}
\item $F: B^2_b \to B^0$ is proper.
\item $\overline u_{nf} = \overline u(\ell,r,\Theta(\ell,r)) \in \overline{{\cal V}}$ is a zero of $\Lambda$ with $f'''(\ell), f'''(r) \ge 0$.
\end{enumerate}
\end{prop}

Hypothesis $(4)$ implies properness of $F$ (\cite{CTZ1}) and thus forces cusps.

\begin{proof} From the hypotheses, a fiber $u(t)$ through  $u(t_0) \in {\cal N}$ reaches arbitrary heights $t$. If there are no cusps, on the other hand, the fibers in ${\cal N} \ne \emptyset$ are taken to a single point, as just shown above. This cannot happen if $F$ is proper.

For the second possibility, we have $\lambda_1(\overline u_{nf})=0$ and $\delta_1(\overline u_{nf}) = 0$, so that $\Pi_W(f''(\overline u_{nf}) \phi_1^2(\overline u_{nf})) = (f''(\overline u_{nf}) \phi_1^2(\overline u_{nf}))$.
 We show $\tau_1(\overline u_{nf})  \ne 0$. By Proposition
\ref{prop:Lambda}, omitting the dependence in $\overline u_{nf}$,
 \[ \tau_1 =  \langle \, \nabla \delta_1 \, , \phi_1 \rangle = \langle - f''' \phi_1^3 - 3 w f'' \phi_1 \, , \, \phi_1 \, \rangle
 \quad \hbox{for} \quad  w = \big( DF  \big)_W^{-1} (f'' \phi_1^2) \, .\]
Now, $ \big( DF  \big)_W^{-1}: \phi_1^\perp \to \phi_1^\perp$ is a positive operator, since from Proposition \ref{prop:lambdacont} only the smallest eigenvalue $\lambda_1$ can be zero. Thus $\langle - 3 w f'' \phi_1 \, , \, \phi_1 \, \rangle < 0 $. We also have $ \langle - f''' \phi_1^3  \, , \, \phi_1 \, \rangle \le 0$, since $f'''(\ell), f'''(r) \ge 0$.

A cusp $u_c$ is obtained by smoothing of $\overline u_{nf}$  once we know the independence of $\nabla \lambda_1(\overline u_{nf})$ and $\nabla \delta_1(\overline u_{nf})$.
Since $\nabla \lambda_1(\overline u_{nf}) \ne 0$ and $\delta_1(\overline u_{nf}) = 0$, this follows from
\[ \tau_1(\overline u_{nf}) = \langle \nabla \delta_1(\overline u_{nf}), \phi_1(\overline u_{nf}) \rangle \ne 0 \, .  \quad  \eqno{\blacksquare} \]

\end{proof}

\bigskip

\section{Proof of Theorem \ref{theo:3sol}} \label{section:cusp2}

Most of the argument follows the proof of Theorem \ref{theo:c}, with simple adaptations. Again, we consider standard boundary conditions. Hypothesis $(H_k)$ guarantees that $\mu_{k,b}$ is simple, so that, from Section \ref{section:spectrumsmooth} the functional $\lambda_k$, the $k$-th eigenvalue of $DF(u)$, and the corresponding normalized eigenfunction $\phi_k$ are well defined and appropriately smooth in a neighborhood $U \subset \Hzb$ of $u_0$ for which $f'(u_0) \equiv  \mu_{k,b}$. The characterization of a cusp $u_c \in U$ is analogous: it is a zero of \[ \Lambda: U \cap B^2_b \to \RR^2 \, , \quad  \Lambda(u) \e (\lambda_k(u), \delta_k(u)  \e \ D \lambda_k(u) \ \phi_k(u) )\]
for which
$D \Lambda(u_c): B^2_b \to \RR^2$ is surjective and $\tau_k(u_c) = D\delta_k(u_c) \, \phi_k(u_c) \ne 0 $.

There are small changes to be made in the counterpart to Proposition \ref{prop:goodnonfold1}.

\begin{prop} \label{prop:goodnonfold2} Suppose $(H_k)$. Then there is a regular nonfold $\overline u_{nf} \in \overline{{\cal V}}$.
\end{prop}

\begin{proof} For the free eigenfunction $\psi_{k,b}$, take $p \in \Omega$ for which $\psi_{k,b}(p) \ne 0$ (this is automatic for $k=1$, since the ground state is positive).

The required $\overline u_{nf}$ is of the form $\overline u = \overline u (\ell = x_{\mu}, r = y_{\mu},\theta) = x_{\mu} \chi_\theta + y_{\mu} \chi_\theta^c$, for some  $\theta = \theta_0$. For {\it any} choice of $\theta$, $f'(\overline u) \equiv \mu_{k,b}$, so that $\lambda_k(\overline u) = 0$ and $f''(\overline u)$ is nonzero on $\chi_\theta$ and $\chi_\theta^c$. For $\theta = 0 $ or $2\pi$, $\overline u$  is constant (either $x_{\mu}$ or $y_{\mu}$) and  $\delta_k(x_{\mu})  \delta_k(y_{\mu}) \le 0$. If it is zero, take the $\overline u$ for which $\delta_k=0$ to be $\overline u_{nf}$ (in this case, $\phi_k^3$ integrates to 0 in $\Omega$). Otherwise, by the $ub$-continuity of $\delta_k$ in Proposition \ref{prop:Lambda}, an intermediate $\theta_0$ yields $\overline u_{nf}$.

We now show regularity. From $f'(\overline u_{nf}) = \mu_{k,b}$, we have
\[ \lambda_k(\overline u_{nf})=0 \, ,\quad   \phi_k(\overline u_{nf}) = \psi_{k,b}\, , \quad DF(\overline u_{nf})=(- \Delta_b- \mu_{k,b}) \, . \]
Also, $\delta_k(\overline u_{nf}) = 0$ means that $\langle \, f''(\overline u_{nf}) \psi_{k,b}^2 \, , \, \psi_{k,b} \rangle =0 $, so that $f''(\overline u_{nf}) \psi_{k,b}^2 \,  \in W$. From  $(H_k)$, $\nabla \lambda_k(\overline u_{nf}) = f''(\overline u_{nf}) \, \psi_{k,b}^2 \ne 0$ which combined with the linear dependence of the gradients $\nabla \lambda_k(\overline u_{nf})$ and $\nabla \delta_k(\overline u_{nf})$ imply collinearity,
\[ \nabla \delta_k (\overline u_{nf}) \, = \, d \ \nabla \lambda_k(\overline u_{nf}), \quad \hbox{for} \ \ d \in \RR \, ,\]
which, as we shall see, leads to a contradiction. By Proposition \ref{prop:Lambda},
\[ - f'''(\overline u_{nf})\,  \psi_{k,b}^3 - 3\,  w\,  f''(\overline u_{nf})\,  \psi_{k,b} \, = \, d \, f'' (\overline u_{nf})\, \psi_{k,b}^2, \]
\[  w \, = \, (DF(\overline u_{nf}) - \lambda_k(\overline u_{nf}))|_W^{-1} \Pi_W \, ( f''(\overline u_{nf}) \, \psi_{k,b}^2 ) \]
\[ =
\, (- \Delta_b- \mu_{k,b})|_W^{-1} \, ( f''(\overline u_{nf}) \psi_{k,b}^2 ) \, \in \, H^2_b \, ,\]
so that (since $f'' \ne 0$ and $\psi_{k,b} = 0$ on a set of measure zero)
\[ - \frac{f'''(\overline u_{nf})}{f''(\overline u_{nf})} \, \psi_{k,b}^2 - 3 \, w \, = \, d \, \psi_{k,b} \, . \]
Now, $w$ and $\psi_{k,b}$ belong to  $H^2_b$, and $\psi_{k,b}(p) \ne 0$, so the two-valued fraction (at $p$, and thus, throughout the jump between $\Omega_{\Theta}$ and $\Omega_{\Theta}^c$) is actually a constant, say $a \in \RR$. Applying $DF(\overline u_{nf})=(- \Delta_b- \mu_{k,b})$,
\[ a \, (- \Delta_b- \mu_{k,b}) \psi_{k,b}^2 - 3 \, f''(\overline u_{nf}) \psi_{k,b}^2 \, = d \, DF(\overline u_{nf}) \psi_{k,b} = 0 \, .\]
The first term on the left hand side is smooth, thus $f''(\overline u_{nf})$ is also a constant. But $f''$ has opposite signs at $x_{\mu}$ and $y_{\mu}$, a contradiction, so that  $\nabla \lambda_k(\overline u_{nf})$ and $\nabla \delta_k(\overline u_{nf})$ are linearly independent in $\Hzb$. \qed
\end{proof}

\medskip
Let $D \subset B^2_b$ be a dense subset of $\Hzb$.
Now mimic Section \ref{section:mollifiers} to obtain $u_{nf} \in D$ by smoothing  $\overline u_{nf}$ and then Proposition \ref{Nfibers} to complete the proof of the dichotomy in  Theorem  \ref{theo:3sol}.

We finally consider the situations in which $F$ necessarily has a cusp. For the hypothesis $k=1$ and $f'''(x_\mu), f'''(y_\mu) \ge 0$, the argument in the proof of the second case of Proposition \ref{prop:noworse} applies with no change.

Suppose then $k=1$ and $F$ proper. We follow the proof of Proposition \ref{Nfibers}, but  there are new difficulties. First notice that under hypothesis $(H_1)$, $DF(u)$ may cease to be invertible and still $\lambda_1(u) \ne 0$. For  the nonfold $\overline u_{nf}$ obtained in the proof of Proposition \ref{prop:goodnonfold2}, however, $\lambda_1(\overline u_{nf}) = 0$.
Also,  $\lambda_1$, $\phi_1 >0 $ and $\delta_1$ are still globally defined and smooth. Let $ \CCC_1 = \{ u \in B^2_b, \lambda_1(u) = 0\} \subset {\cal C}$. Since $\nabla \lambda_1(u) \ne 0$ for $u \in  \CCC_1$ (as in the proof of the proposition above), $ \CCC_1$ is a manifold. It is easy to see that some smoothing $u_{nf}$ obtained from $\overline u_{nf}$ belongs to $\CCC_1$. Define $\DDD_1$ to be the zero level of the functional $\delta_1$.

Following the proof of Proposition \ref{prop:taubasta}, every nonfold $u_{nf} \in {\cal N}_1 =  \CCC_1 \cap \DDD_1$ is regular. Again, ${\cal N}_1$ is the transversal intersection of the two manifolds $ \CCC_1$ and $\DDD_1$ near ${\cal N}_1$ as in  Proposition \ref{Nfibers}. If $u \in {\cal N}_1$ and $\tau_1(u) \ne 0$, then  $u$ is a cusp of $F$. Suppose that $\tau_1 \equiv 0$  in ${\cal N}_1$ --- we  derive a contradiction.

As in Proposition \ref{Nfibers}, the vector field $u \mapsto \tilde \phi_1(u)$ leaves ${\cal N}_1$ invariant, and each integral curve $\gamma = \{ u(t), t \in \RR \} \subset {\cal N}_1$ of $u' = \tilde \phi_1(u)$  is sent by $F$ to a single point, $F(\gamma)$. These integral curves may not be periodic. Indeed, the function $t \mapsto \langle \psi_{1,b} , u(t) \rangle$ is a height function along the curve, since it has positive derivative: for $k=1$, the eigenfunctions $\psi_{1,b}$ and $\phi_1(u(t))$ are positive. Take the integral curve $\gamma_{u_0}=\{ u(t), t \in \RR , u(0) = u_0 \} \subset {\cal N}_1$  with $F(\gamma_{u_0}) = z_0$.

The set $F^{-1}(z_0) \cap {\cal N}_1 $ may possibly disconnect but it still locally an interval: more precisely, for each point $y_0 \in F^{-1}(z_0) \cap {\cal N}_1$  there is an open neighborhood $U_{y_0} \subset B^2_b$ of $y_0$ so that $F^{-1}(z_0) \cap U_{y_0}$ is an arc $\tilde\gamma_{y_0}$ (i.e., $\tilde\gamma_{y_0}$ is diffeomorphic to an open interval). This follows from a local form of the construction of fibers. Write a Lyapunov-Schmidt decomposition
\[ B^2_b =  W \oplus \langle \phi_1(y_0) \rangle \, , \quad  B^0 = \Ran DF(y_0) \oplus \langle \phi_1(y_0) \rangle \]
where $\langle \phi_1(y_0) \rangle $ is the line spanned by $\phi_1(y_0)$. For small $w \in W$, $t \in \RR$, split
\[ F (y_0 + w + t \phi_1 (y_0)) \, = \, \Pi \, F (y_0 + w + t \phi_1 (y_0)) + (I - \Pi ) \, F (y_0 + w + t \phi_1 (y_0)) \, ,\]
where $\Pi: B^0 \to \Ran DF(y_0)$ is the projection with $\ker \Pi = \langle \phi_1(y_0) \rangle$. From simple spectral arguments, the inverse function theorem applies and we learn that, for each fixed $t \sim 0$, the map $w \mapsto \Pi F (y_0 + w + t \phi_1(y_0))$ is a local diffeomorphism near $w=0$ to $\Ran DF(y_0)$. In particular, the inversion of a small segment $z_0 + t\langle \phi_1(y_0) \rangle, t \sim 0$, obtains the small isolated arc $\tilde\gamma_{y_0} \in B^2_b \cap U_{y_0}$ through $y_0$ --- in a sense, this arc is a local chunk of a fiber.

Take now $y_0 = u_0$ in the construction above and we have, for $F(u_0) = z_0$,
\[ \gamma_{u_0} \cap U_{u_0} \subset \tilde\gamma_{u_0} = F^{-1}(z_0) \cap U_{u_0} \, . \]

Now, suppose that the closure $\overline \gamma_{u_0} \subset F^{-1}(z_0)$ is compact. Then, at integer times $t=n$, the sequence $u(n)$ accumulates to $u_\ast \in B^2_b$. By continuity, since $u(n) \in {\cal N}_1$, we have that $u_\ast \in {\cal N}_1$ and $F(u_\ast) = z_0$.
Consider the solution $\gamma_{u_{\ast}}$ of $u' = \tilde \phi(u), u(t_0) = u_{\ast}$. For large $n$,
\[ u(n) \in \gamma_{u_\ast} \subset \tilde\gamma_{u_\ast} \subset F^{-1}(z_0) \cap U_{u_\ast} \, . \]

In particular, the points $u(n)$ lie in the single arc through their accumulation $u_\ast \in \gamma_{u_\ast} \cap U_{u_\ast}$ for some neighborhood $U_{u_\ast}$. More, since the tangent vector $\tilde \phi(u)$ is of norm one, the points $u(n)$ are far apart along the orbit. Accumulation is not possible then, since the orbit admits a height function: $\overline \gamma$ is not compact.

Finally, $F^{-1}(z_0)$   is also not compact (otherwise the closed subset $\overline \gamma$ would be too), contradicting the properness of $F$. The proof of Theorem \ref{theo:3sol} is now complete.

\section{Appendix: well definedness and continuity} \label{appendix:continuity}

We  prove that the basic functions in the text satisfy the required continuity and differentiability. We assume standard boundary conditions.

\subsection{The eigenpair $(\lambda^q, \phi^q)$ for potentials $q \in L^\infty$} \label{section:spectrumsmooth}

We use a result in \cite{Z} yielding smoothness of eigenvalues and eigenvectors in the case of interest.  For $X \subset Y$  real Banach spaces, let ${\cal B} = {\cal B}(X,Y)$ be the Banach space of bounded linear transformations from $X$ to $Y$, with the operator norm.

\medskip

\begin{prop} \label{prop:abstract} Let $T_0 \in {\cal B}$ have eigenvalue $\lambda_0 \in \RR$ and eigenvector $\phi_0 \in X$, so that $(T_0 - \lambda_0 I)\phi_0=0$. Assume that $T_0 - \lambda_0 I$ is a Fredholm operator of index zero with one dimensional kernel, and that $\phi_0 \notin \Ran (T_0 - \lambda_0 I)$. Let $\ell \in X^\ast$ be a linear functional for which $\ell(\phi_0)=1$ and set $V_2 = \phi_0 + \Ker \ell$. Then there is an open neighborhood $U \subset {\cal B}$ of $T_0$ and unique analytic maps $\lambda:U \to \RR$ and $\phi: U \to V_2$ for which $(T - \lambda(T)I)\, \phi(T) = 0$ and $\lambda(T_0) = \lambda_0, \, \phi(T_0)=\phi_0$.
\end{prop}

The operators $T_q$ of Proposition \ref{prop:lambdacont} are smooth functions of $q \in L^\infty(\Omega)$: the linear map taking $q$  to `multiplication by $q$'
%
\[
\Phi_q :  L^\infty  \to  {\cal B}(H^2_b, H^0) \ , \quad
  q  \mapsto  M_q
\]
\noindent is clearly bounded.   For $q \in L^\infty$, Propositions \ref{prop:lambdacont} and \ref{prop:smooth} implies  the hypothesis of the proposition above, and thus   $\lambda^q_1: L^\infty \to \RR$ and $\phi^q_1: L^\infty \to H^0$ are smooth.

\subsection{The eigenpair $(\lambda, \phi)$ as a function of $u \in \Hzb$} \label{subsection:spectrumu}

Let $f: \RR \to \RR$ be smooth, $B^2_b = H^2_b \cap C^{2,\alpha}(\Omega) $.  The function $F: B^2_b \to B^0$ is smooth from Proposition \ref{prop:smooth} and we are interested in $\lambda(u)$ and $\phi(u)$, a simple eigenvalue and corresponding normalized eigenvector of $DF(u)= - \Delta_b- f'(u) $ for the potential $q = f ' (u)$. The functions $\lambda(u)$ and $\phi(u)$ can be either $\lambda_1 = \lambda^q_1$ and $\phi_1$ or $\lambda_k= \lambda_k^q$ and $\phi_k$, depending on the context. The $L^2$ inner product $\langle u, v \rangle$ makes sense for functions $u, v \in B^0$. Take $\ell(v) = \langle \phi(u), v \rangle$, so that $\ell(\phi(u)) = 1$.

\bigskip
\noindent  {\bf Proof of Proposition \ref{prop:fantasma}:}
Proposition \ref{prop:abstract} implies the smoothness of $\lambda$ and $\phi$ for $u \in B^2_b$: the smoothness of  $u \in B^2_b \mapsto f'(u) \in B^0$ is (easily) proved.

We now consider $u \in \Hzb$.
We show that $\Phi : \Hzb \to {\cal B}(H^2_b, H^0) , \,
u \mapsto M_{f'(u)}$ is $ub$-continuous.
Take $v \in H^2_b, \| v \|=1$. Then
\[ \| (f'(u_m) - f'(u_\infty) )\  v \|_{H^0} \le \| f'(u_m) - f'(u_\infty) \|_{L^{2r}} \| v \|_{L^{2s}}, \quad 1/r + 1/s = 1.\]
Suppose now $n \ge 5$. Take $r= n/2, s = n/(n-2) > 1$ so that $H^2_b$ imbeds continuously in $L^{2s}$ and $ \| v \|_{L^{2s}} \le \CCC \| v \|_{H^2_b} $.
The result follows from the dominated convergence theorem because of the uniform $L^\infty$ bound on $f'(u_m)$ and $f'(u_\infty)$. For $n \le 4$, imbed $H^2_b$  in some $L^q$ for some $q >2$ and repeat.
To finish the proof,  compose $ u \mapsto M_{f'(u)}$ with the map from bounded potentials to eigenpairs. A second normalization yielding $L^2$-normal eigenvectors is clearly a smooth map.
\qed

\medskip
The formulas for the derivatives of $\lambda$ and $\phi$ at a point $u$ along a direction $v$ are familiar  \cite{Lax}. We confirm their validity for the more unusual scenario $u \in \Hzb$.

  \bigskip
\noindent {\bf Proof of Proposition \ref{prop:derlambda}:} For the directional derivatives $D \lambda(u) v$, subtract
\[ DF(u + tv)  \phi(u + tv)= \lambda(u + tv) \phi(u + tv) \, , \quad
DF(u )  \phi(u)= \lambda(u) \phi(u) \]
to obtain, denoting differences $g(u+tv) - g(u)$ by $Sg$,
\[ \langle DF(u + tv)S\phi \, , \, \phi(u+tv)\rangle
+ \langle SDF \, \phi(u), \, \phi(u+tv) \rangle = \]
\[ \langle \lambda(u + tv)S\phi \, , \, \phi(u+tv)\rangle
+ \langle S\lambda \,  \phi(u), \, \phi(u+tv) \rangle.\]
From the symmetry of $DF(u+tv)$, the first terms on each side  cancel each other.
We now take limits and use the continuity of $\phi: \Hzb \to H^0$: for $u, v \in \Hzb$,
\[ \lim_{t\to 0} \frac{1}{t}(\lambda(u+tv) - \lambda(u)) \langle \phi(u), \phi(u) \rangle =  \lim_{t\to 0} \langle \frac{1}{t}(DF(u+tv)- DF(u)) \, \phi(u), \, \phi(u) \rangle \]
and setting $\nabla \lambda(u) = - f''(u) \phi^2(u) \in \Hzb$, by the dominated convergence theorem,
\[ D \lambda(u) v = \lim_{t\to 0} \, - \langle \, \frac{1}{t} \, (f'(u+tv) - f'(u)) \, \phi(u) \, , \, \, \phi(u) \, \rangle = \langle \nabla \lambda(u) , v \rangle\, . \]
For the $ub$-continuity of $\nabla \lambda(u): \Hzb \to \Hzb$, take $u_m \ub u_\infty$  with  $\|u_m \|_\infty  \le {\cal C}$: we show that both terms go to zero in
\[ \| f''(u_m) \phi^2(u_m) - f''(u_\infty) \phi^2(u_m) \| + \| f''(u_\infty) \phi^2(u_m) - f''(u_\infty) \phi^2(u_\infty) \| \, .\]
The new ingredient is the uniform bound of the sequence $\{\phi(u_m)\}$ from Proposition  \ref{prop:lambdacont}. The first term goes to zero because $f '' $ is Lipschitz on $[-C, C]$.
\qed

%

\bigskip
Consider $W^2 \subset H^2_b,W^0 \subset H^0$ the subspaces of functions orthogonal to $\phi(u)$. Let $\Pi_{W }$ be the orthogonal projection from $B^0$ to $W^0$.

\begin{lemma} \label{lemma:phi} $\phi: \Hzb \to \RR$ admits Gateaux derivatives along  functions $v \in \Hzb$,
\[ D\phi(u) \, v =   (DF(u) - \lambda(u) I)_W^{-1}\Pi_W (f''(u)\,  \phi(u) \, v) \in H^0 \, . \]
Also, $D \phi : \Hzb \times \Hzb \to H^0 \, , D\phi(u,v) = D\phi(u) \, v,$ is $ub$-continuous.
\end{lemma}

\begin{proof}
Let $u, v, \in \Hzb$. Start as in the proof above to obtain
\[ \big( \, SDF(u )  - S\lambda(u) \, \big)  \,  \phi(u+tv)   + (DF(u) - \lambda(u) ) \, S\phi(u)  = 0 \, . \]
After dividing by $t$ and taking $t \to 0$, the first term converges to
\[ - f''(u)\,  \phi(u) \, v + \langle f''(u) \phi^2(u), v \rangle \, \phi(u) = - \Pi_W (f''(u)\,  \phi(u) \, v )\, ,\]
since $\Pi_W$ projects orthogonally. The restriction
$(DF(u) - \lambda(u) I)_W: W^2 \to W^0$ is an isomorphism. The derivative of $\phi$, in the second term,
exists and satisfies
\[ D\phi(u) \, v =   (DF(u) - \lambda(u) I)_W^{-1} \, \Pi_W \, (f''(u)\,  \phi(u) \, v)  \in W^2\, .\]

The spaces $W^2$ and $W^0$ depend on $u$, being orthogonal complements of $\phi(u)$: an  algebraic argument  clarifies continuity. Define $\tilde T = \tilde T(u): H^2_b \to H^0$ as
\[ \tilde T =  DF(u) - \lambda(u) I + \phi(u) \otimes \phi(u) \, \quad    \hbox{where } \, (\phi(u) \otimes \phi(u)) z = \langle \phi(u), z \rangle \, \phi(u)  \, . \]
Notice that $\tilde T$ leaves $W^2$ and $\langle \phi(u) \rangle$ invariant. Also, the restrictions to $W^2$ of $(DF(u) - \lambda(u) I)$ and $\tilde T$ coincide and are invertible. Since $\tilde T \phi(u) = \phi(u)$, $\tilde T$ is invertible. As  in the proof of the $ub$-continuity of $\lambda$, $\tilde T \in {\cal B}(H^2_b, H^0)$ varies $ub$-continuously in $u \in \Hzb$ . Inversion preserves continuity and thus
\[ D\phi(u) \, v \ = \ \Pi_W \ C^{-1} \ \Pi_W \, (f''(u) \phi(u) \, v ) \, ,\]
 is $ub$-continuous, as well as
\[ u \in  H^2_b \mapsto \Pi_W(f''(u) \phi^2(u)) =    f''(u)\,  \phi^2(u) \,  + \langle f''(u) \phi^2(u), \phi(u) \rangle \, \phi(u) \in H^0 . \quad \blacksquare\]
\end{proof}

\subsection{The functionals $\delta$, $\tau$ and the function $\Lambda$} \label{section: Lambda}
\medskip
Differentiability properties for $\Lambda : \Hzb \to \RR^2$ require equivalent statements for $\lambda$ (Proposition \ref{prop:derlambda}) and for $\delta: \Hzb \to \RR$, which we prove now.

\bigskip
\noindent{\bf Proof of Proposition \ref{prop:Lambda}:} Since $\delta (u)= \langle \nabla \lambda(u) , \phi(u) \rangle$, $ub$-continuity of $\delta$ follows from Propositions \ref{prop:fantasma} and \ref{prop:derlambda}.
We now take directional derivatives $D\delta(u)v$:
\[ D \delta(u) \, v = - \int_\Omega f'''(u) \, \phi^3(u) \, v - \int_\Omega f''(u) \, 3 \, \phi^2(u) \, (D\phi(u) \, v ) \, . \]
On the second term, using $ D\phi(u) \, v \ = \ \Pi_W \ C^{-1} \ \Pi_W \, (f''(u) \phi(u) \, v )$, we have
\[ \langle \, f''(u) \,3 \,\phi^2(u)\, ,\,  D\phi(u) \, v  \, \rangle  = \, \langle \, \Pi_W \, C^{-1} \, \Pi_W \, (f''(u) \, 3 \, \phi^2(u)\, ) ,   \, (f''(u) \, \phi(u) \, v ) \, \rangle \]
\[ = \, \langle \, \Pi_W \, \big( \, DF(u) - \lambda(u) I \, \big)_W^{-1} \, \Pi_W \, (f''(u)\,  3 \, \phi^2(u)),   \, (f''(u) \, \phi(u)\,  v )\,  \rangle \]
\[ =
\langle 3 \, w(u),   \, (f''(u) \, \phi(u) \, v ) \rangle \]
and the computation of $D\delta(u)v$ is complete: we are left with showing the continuity of $\nabla \delta(u)$. For $ u_m \, \ub \ u_\infty $, we show the $L^2$ convergences
\[ -f'''(u_m)\phi^3(u_m)  \to -f'''(u_\infty)\phi^3(u_\infty) \, ,  \]
\[  w(u_m) \, f''(u_m) \phi(u_m) \to  3 w(u_\infty) \, f''(u_\infty) \phi(u_\infty) \, .  \]
For the first term, proceed as in the argument for $\nabla \lambda$. For the second, we show
\[\| w(u_m) \big( \, f''(u_m) \phi(u_m) - f''(u_\infty) \phi(u_\infty)\big) \| +
\| \big( w(u_m) - w(u_\infty) \big) f''(u_\infty) \phi(u_\infty)\big) \| \to 0 \, . \]
We first prove that  $w(u_m) \to w(u_\infty)$ in $L^2$.

Since $f''(u_\infty) \phi(u_\infty)$ is bounded,  $\| \big( w(u_m) - w(u_\infty) \big) f''(u_\infty) \phi(u_\infty)\big) \| \to 0$. For the first term, split again: we show that
\[\| \big(w(u_m) - w(u_\infty) \big) \big( \, f''(u_m) \phi(u_m) - f''(u_\infty) \phi(u_\infty)\big) \| \] and
\[ \| w(u_\infty) \big( \, f''(u_m) \phi(u_m) - f''(u_\infty) \phi(u_\infty)\big) \|\]
go to $0$.
As before, $w(u_m) \to w(u_\infty)$ and $\, f''(u_m) \phi(u_m) - f''(u_\infty) \phi(u_\infty)$ are uniformly bounded, and the first term is done. The second follows by the dominated convergence theorem.


\medskip
The fact that $\Lambda|_{V_\ast}$ is $\CCC^1$ on finite dimensional subspaces $V_\ast$ follows from the statement just proved: its partial derivatives are continuous. One might  use the sup norm in the arguments: on $V_\ast$, the $L^2$ and the $L^\infty$ norms are equivalent.\qed

{

\parindent=0pt
\parskip=0pt
\obeylines

\bigskip

Marta Calanchi, Dipartimento di Matematica, Università di Milano,
Via Saldini 50, 20133 Milano, Italia

\smallskip

Carlos Tomei and André Zaccur, Departamento de Matem\'atica, PUC-Rio,
R. Mq. de S. Vicente 225, Rio de Janeiro, RJ 22453-900, Brazil

\bigskip

marta.calanchi@unimi.it
carlos.tomei@gmail.com
zaccur.andre@gmail.com
}

\end{document}